\documentclass{amsart}
\usepackage{amsmath}
\usepackage{amssymb}
\usepackage{amsfonts}

\theoremstyle{plain}
\newtheorem{acknowledgement}{Acknowledgement}
\newtheorem{algorithm}{Algorithm}

\newtheorem{claim}{Claim}

\newtheorem{definition}{Definition}
\newtheorem{example}{Example}

\newtheorem{proposition}{Proposition}
\newtheorem{remark}{Remark}

\numberwithin{equation}{section}

\begin{document}
\title[]{Finite representations of the braid group commutator subgroup}
\author{Abdelouahab AROUCHE}
\address{USTHB, Fac. Math. P.O.Box 32 El Alia 16111 Bab Ezzouar Algiers,
Algeria}
\email{abdarouche@hotmail.com}
\date{April 06, 2007.}
\subjclass[2000]{Primary 20F36, 20C40; Secondary 20E07.}
\keywords{Braid group, commutator subgroup, representation, symmetric group,
special linear group. }

\begin{abstract}
We study the representations of the commutator subgroup $K_{n}$ of the braid
group $B_{n}$ into a finite group $\Sigma $. This is done through a symbolic
dynamical system. Some experimental results enable us to compute the number
of subgroups of \ $K_{n}$ of a given (finite) index, and, as a by-product,
to recover the well known fact that every representation of $K_{n}$\ \ into $%
S_{r}$\ , with $n>r$ , must be trivial.
\end{abstract}

\maketitle

\section{Introduction}

In [SiWi1], D. Silver and S. Williams exploited the structure of the kernel
subgroup $K$ of an epimorphism $\chi :G\rightarrow
%TCIMACRO{\U{2124} }%
%BeginExpansion
\mathbb{Z}
%EndExpansion
$, where $G$ is a finitely presented group, to show that the set $Hom\left(
K,\Sigma \right) $ of representations of $K$ into a finite group $\Sigma $
has a structure of a subshift of finite type (SFT), a symbolic dynamical
system described by a graph $\Gamma $; namely, there is a one to one
correspondence between representations $\rho :K\rightarrow \Sigma $ and
bi-infinite paths in $\Gamma $.

We apply this method to the group $B_{n}$ of braids with $n$-strands, with $%
\chi $ being the abelianization homomorphism and $\Sigma $ the symmetric
group $S_{r}$ of degree $r$ or the special linear group over a finite field $%
SL_{2}\left( F_{q}\right) $. The subgroup $K_{n}=\ker \chi $ is then the
commutator subgroup \ of $B_{n}$.

\bigskip It is a well known fact that for a given group $K$, there is a
finite to one correspondence between its subgroups of index no greater than $%
r$ and representations $\rho :K\rightarrow S_{r}$. This correspondence can
be described by
\begin{equation*}
\rho \longmapsto \left\{ g\in K:\rho \left( g\right) \left( 1\right)
=1\right\} .
\end{equation*}%
The pre-image of a subgroup of index exactly $r$ consists of $\left(
r-1\right) !$ transitive representations $\rho $. ($\rho $ is said to be
transitive if $\rho \left( K\right) $ operates transitively on $\left\{
1,2,...,r\right\} $). This will allow us to draw some conclusions about the
subgroups of finite index of $K_{n}$, which motivates the choice of \ $S_{r}$%
; we motivate that of $SL_{2}\left( F_{q}\right) $ by the fact that any
representation of $K_{n}$ into any group $\Sigma $ has range in the
commutator subgroup $\left[ \Sigma ,\Sigma \right] $, for $n\geq 6$.

We give an algorithm to compute $Hom\left( K_{n},\Sigma \right) $, for $%
n\geq 5$ . Some experimental results enable us to compute the number of
subgroups of \ $K_{n}$ of a given (finite) index, and, as a by-product, to
recover the well known fact that $Hom\left( K_{n},S_{r}\right) $ is trivial
for $n\geq 5$ and $r<n$ . Since every representation in $Hom\left(
B_{n},\Sigma \right) $ restricts to an element of $Hom\left( K_{n},\Sigma
\right) $, we enhance the given algorithm in order to compute $Hom\left(
B_{n},\Sigma \right) $.

\section{Generalities}

Let $B_{n}$ be the braid group given by the presentation:

\begin{equation*}
\langle \sigma _{1},...,\sigma _{n-1}\left\vert
\begin{array}{cc}
\sigma _{i}\sigma _{j}=\sigma _{j}\sigma _{i}; & \left\vert i-j\right\vert
\geq 2 \\
\sigma _{i}\sigma _{j}\sigma _{i}=\sigma _{j}\sigma _{i}\sigma _{j}; &
\left\vert i-j\right\vert =1%
\end{array}%
\right. \rangle ,
\end{equation*}

(see [BuZi] for additional background). Let $\beta \in B_{n}$ be a braid.
Then $\beta $ can be written as :
\begin{equation*}
\beta =\sigma _{i_{1}}^{\varepsilon _{1}}\cdot \cdot \cdot \sigma
_{i_{k}}^{\varepsilon _{k}},
\end{equation*}%
with $i_{1},\cdot \cdot \cdot ,i_{k}\in \left\{ 1,\cdot \cdot \cdot
,n-1\right\} $ and $\varepsilon _{i}=\pm 1$. Define the exponent sum of $%
\beta $ (in terms of the $\sigma _{i}$'s) denoted by $\exp (\beta )$, as:%
\begin{equation*}
\exp (\beta )=\varepsilon _{1}+\cdot \cdot \cdot +\varepsilon _{k}.
\end{equation*}%
Then $\exp (\beta )$ is an invariant of the braid group, that is, it doesn't
depend on the writing of $\beta $. Moreover, $\exp (\beta ):B_{n}\rightarrow
%TCIMACRO{\U{2124} }%
%BeginExpansion
\mathbb{Z}
%EndExpansion
$ is a homomorphism. Let $H_{n}$ denote its kernel. Then the
Reidemeister-Schreier theorem [LySc] enables us to find a presentation for $%
H_{n}$. We choose the set
\begin{equation*}
\left\{ \cdot \cdot \cdot ,\sigma _{1}^{-m},\sigma _{1}^{-m+1},\cdot \cdot
\cdot ,\sigma _{1}^{-1},1,\sigma _{1},\sigma _{1}^{2},\cdot \cdot \cdot
,\sigma _{1}^{m},\cdot \cdot \cdot \right\}
\end{equation*}%
as a Schreier system of right coset representatives of $H_{n}$ in $B_{n}$.
Putting $z_{m}=\sigma _{1}^{m}\left( \sigma _{2}\sigma _{1}^{-1}\right)
\sigma _{1}^{-m}$ for $m\in
%TCIMACRO{\U{2124} }%
%BeginExpansion
\mathbb{Z}
%EndExpansion
$, and $x_{i}=\sigma _{i}\sigma _{1}^{-1}$ for $i=3,\cdot \cdot \cdot ,n-1$,
we get the following presentation of $H_{n}:$

\begin{equation*}
H_{n}=\langle
\begin{array}{cc}
z_{m}, & m\in
%TCIMACRO{\U{2124} }%
%BeginExpansion
\mathbb{Z}
%EndExpansion
\\
x_{i}, & i=3,\cdot \cdot \cdot ,n-1%
\end{array}%
\left\vert
\begin{array}{cc}
x_{i}x_{j}=x_{j}x_{i}, & \left\vert i-j\right\vert \geq 2; \\
x_{i}x_{j}x_{i}=x_{j}x_{i}x_{j}, & \left\vert i-j\right\vert =1; \\
\begin{array}{c}
z_{m}z_{m+2}=z_{m+1}, \\
z_{m}x_{3}z_{m+2}=x_{3}z_{m+1}x_{3},%
\end{array}
& m\in
%TCIMACRO{\U{2124} }%
%BeginExpansion
\mathbb{Z}
%EndExpansion
; \\
z_{m}x_{i}=x_{i}z_{m+1}, &
\begin{array}{c}
i=4,\cdot \cdot \cdot ,n-1; \\
m\in
%TCIMACRO{\U{2124} }%
%BeginExpansion
\mathbb{Z}
%EndExpansion
.%
\end{array}%
\end{array}%
\rangle \right.
\end{equation*}

\begin{example}
We have:%
\begin{equation*}
H_{3}=\langle z_{m}\left\vert z_{m}z_{m+2}=z_{m+1};\forall m\in
%TCIMACRO{\U{2124} }%
%BeginExpansion
\mathbb{Z}
%EndExpansion
\right. \rangle
\end{equation*}%
is a free group on two generators $z_{0}=\sigma _{2}\sigma _{1}^{-1}$ and $%
z_{-1}=\sigma _{1}^{-1}\sigma _{2}$, and
\begin{equation*}
H_{4}=\langle z_{m},t\left\vert
\begin{array}{cc}
\begin{array}{c}
z_{m}z_{m+2}=z_{m+1}, \\
z_{m}tz_{m+2}=tz_{m+1}t,%
\end{array}
& m\in
%TCIMACRO{\U{2124}}%
%BeginExpansion
\mathbb{Z}%
%EndExpansion
\end{array}%
\right. \rangle .
\end{equation*}%
Note that $H_{2}=\left\{ 1\right\} $, since $B_{2}=\langle \sigma
_{1}\left\vert -\right. \rangle \cong
%TCIMACRO{\U{2124} }%
%BeginExpansion
\mathbb{Z}
%EndExpansion
$.
\end{example}

Now, every commutator in $B_{n}$ has exponent sum zero. Conversely, every
generator of $H_{n}$ \ is a product of commutators. Hence, we have $%
H_{n}=K_{n}$, and $\exp $ is the abelianization homomorphism.

Each $K_{n}$ fits into a split exact sequence:%
\begin{equation*}
1\rightarrow K_{n}\rightarrow B_{n}\rightarrow
%TCIMACRO{\U{2124} }%
%BeginExpansion
\mathbb{Z}
%EndExpansion
\rightarrow 0,
\end{equation*}%
and there are "natural" inclusions $B_{n}\subset B_{n+1}$ and $K_{n}\subset
K_{n+1}$.

\begin{remark}
We have the following consequences of relations in $K_{n}$ :
\end{remark}

\begin{enumerate}
\item $z_{m+1}=x_{j}^{-1}z_{m}x_{j}$, hence $z_{m}=x_{j}^{-m}z_{0}x_{j}^{m}$%
; for $j\geq 4$. where $z_{0}=\sigma _{2}\sigma _{1}^{-1}$.

\item The $z_{m}$'s are conjugate in $B_{n}$ for $n\geq 3$ (by $\sigma _{1}$%
) and in $K_{n}$ for $n\geq 5$.

\item The restriction of conjugation by $\sigma _{1}$ in $B_{n}$ to $K_{n}$
induces an action of $%
%TCIMACRO{\U{2124} }%
%BeginExpansion
\mathbb{Z}
%EndExpansion
$ on $K_{n}$. This presentation of $K_{n}$ is said to be $%
%TCIMACRO{\U{2124} }%
%BeginExpansion
\mathbb{Z}
%EndExpansion
$-dynamical.
\end{enumerate}

\section{The representation shift}

This work is essentially experimental. It aims to describe the set of
representations of $K_{n}$ into finite group $\Sigma .$ We start with $n=3$
and describe $Hom\left( K_{3},\Sigma \right) $ by means of a graph $\Gamma $
that we will construct in a step by step fashion [SiWi2]. A representation $%
\rho :K_{3}\rightarrow \Sigma $ is a function $\rho $ from the set of
generators $z_{m}$ to $\Sigma $ such that for each $m\in
%TCIMACRO{\U{2124} }%
%BeginExpansion
\mathbb{Z}
%EndExpansion
$, the relation:%
\begin{equation*}
\rho \left( z_{m}\right) \rho \left( z_{m+2}\right) =\rho \left(
z_{m+1}\right)
\end{equation*}%
holds in $\Sigma $. Any such function can be constructed as follows,
beginning with step $0$ and proceeding to steps $\pm 1,\pm 2,\cdot \cdot
\cdot $

$\cdot $

$\cdot $

$\cdot $

(step -1) Choose $\rho \left( z_{-1}\right) $ if possible such that $\rho
\left( z_{-1}\right) \rho \left( z_{1}\right) =\rho \left( z_{0}\right) $.

(step 0) Choose values $\rho \left( z_{0}\right) $ and $\rho \left(
z_{1}\right) $.

(step 1) Choose $\rho \left( z_{2}\right) $ if possible such that $\rho
\left( z_{0}\right) \rho \left( z_{2}\right) =\rho \left( z_{1}\right) $.

(step 2) Choose $\rho \left( z_{3}\right) $ if possible such that $\rho
\left( z_{1}\right) \rho \left( z_{3}\right) =\rho \left( z_{2}\right) $.

$\cdot $

$\cdot $

$\cdot $

This process leads to a bi-infinite graph whose vertices are the maps $\rho
:\left\{ z_{0},z_{1}\right\} \rightarrow \Sigma $, each of which can be
regarded as an ordered pair $\left( \rho \left( z_{0}\right) ,\rho \left(
z_{1}\right) \right) $. There is a directed edge from $\rho $ to $\rho
^{\prime }$ if and only if $\rho \left( z_{1}\right) =\rho ^{\prime }\left(
z_{0}\right) $ and $\rho \left( z_{0}\right) \rho ^{\prime }\left(
z_{1}\right) =\rho \left( z_{1}\right) $. In such a case, we can extend $%
\rho :\left\{ z_{0},z_{1}\right\} \rightarrow \Sigma $ by defining $\rho
\left( z_{2}\right) $ to be equal to $\rho ^{\prime }\left( z_{1}\right) $.
Now if there is an edge from $\rho ^{\prime }$ to $\rho "$, we can likewise
extend $\rho $ by defining $\rho \left( z_{3}\right) $ to be $\rho "\left(
z_{1}\right) $. We implement this process by starting with an ordered pair $%
\left( a_{0},a_{1}\right) $ of elements of $\Sigma $, and computing at each
step a new ordered pair from the old one, so that every edge in the graph
looks like:%
\begin{equation*}
\left( a_{m},a_{m+1}\right) \rightarrow \left( a_{m+1},a_{m+2}\right) ,
\end{equation*}%
with
\begin{equation*}
a_{m+2}=a_{m}^{-1}a_{m+1}.
\end{equation*}%
In our case, since the group $\Sigma $ is finite, the process must end, and
the graph $\Gamma $ we obtain consists necessarily of disjoint cycles. This
gives an algorithm for finding $Hom(K_{3},\Sigma )$. Observe that $%
Hom(K_{3},\Sigma )$ is endowed with a shift map%
\begin{equation*}
\sigma :\rho \longmapsto \sigma \left( \rho \right)
\end{equation*}%
defined by
\begin{equation*}
\sigma \left( \rho \right) :x\longmapsto \rho \left( \sigma _{1}x\sigma
_{1}^{-1}\right) .
\end{equation*}%
If we regard $\rho $ as a bi-infinite path in the graph $\Gamma $, then $%
\sigma $ correspond to the shift map $\left( a_{m},a_{m+1}\right)
\longmapsto \left( a_{m+1},a_{m+2}\right) $, since $z_{m+1}=\sigma
_{1}z_{m}\sigma _{1}^{-1}$. Any cycle in the graph $\Gamma $ with length $p$
corresponds to $p$ representations having least period $p$. These are the
iterates of some representation $\rho \in Hom(K_{3},\Sigma )$ satisfying $%
\rho \left( z_{m}\right) =a_{m}$ and $\sigma ^{p}\left( \rho \right) =\rho $%
, since $a_{m+p}=a_{m}$.

\begin{example}
A cycle of length $2$ has the form $\left( a,a^{2}\right) \leftrightarrows
\left( a^{2},a\right) $, with $a^{3}=1$.
\end{example}

\begin{remark}
Since $K_{3}$ is a free group of rank 2, $Hom\left( K_{3},\Sigma \right) $
has a simple description; namely, there is a one to one correspondence
between $Hom\left( K_{3},\Sigma \right) $ and $\Sigma ^{2}$. The interest of
our approach, beside the dynamical structure, is that it allows to go
further and describe $Hom\left( K_{n},\Sigma \right) $, for $n\geq 3$. As a
consequence of the dynamical approach, we get a partition of $Hom\left(
K_{3},\Sigma \right) $ into orbits, hence the identity $\left\vert \Sigma
\right\vert ^{2}=1+\underset{p\geq 2}{\sum }p.n_{p}$, where $n_{p}$ is the
number of orbits of (least) period $p$.
\end{remark}

\begin{proposition}
If a cycle has length $p$, then the identity $a_{0}a_{1}...a_{p-1}=1$ holds.
\end{proposition}

In order to minimize calculations, we extract some foreseeable behaviour for
various choices of the ordered pair $\left( a_{0},a_{1}\right) $ in the
previous algorithm.

First, the dynamics of ordered pairs $\left( a_{0},a_{1}\right) $ such that $%
a_{0}=1$ or $a_{1}=1$ or $a_{0}=a_{1}$ is entirely known. To be precise, we
get a cycle of length $6$ unless $a=a^{-1}$, in which case it is of length $%
3 $ (or $1$ if and only if $a=1$).
\begin{eqnarray*}
\left( 1,a\right) &\rightarrow &\left( a,a\right) \rightarrow \left(
a,1\right) \rightarrow \left( 1,a^{-1}\right) \\
&\rightarrow &\left( a^{-1},a^{-1}\right) \rightarrow \left( a^{-1},1\right)
\rightarrow \left( 1,a\right) .
\end{eqnarray*}%
This sort of dynamics will be generalized later to representations into
abelian groups. Second, when we proceed to a new step, we do not need to
take a pair we have already got in a previous cycle, since we would get
indeed the same cycle. The following dichotomy will prove useful in the
sequel:

\begin{definition}
If a vertex of a cycle in $\Gamma $ has equal components, then the cycle is
said to be of type I. Otherwise, it is of type II.
\end{definition}

Note that a cycle is determined by any of its vertices. Furthermore, the
type I cycles are determined by elements of $\Sigma $.

\bigskip Now let us proceed to compute $Hom(K_{4},\Sigma )$. Since $%
K_{3}\subset K_{4}$, every representation $\rho \in Hom(K_{4},\Sigma )$
restricts to a representation $\rho \left\vert _{K_{3}}\right. \in
Hom(K_{3},\Sigma )$, the latter being described by a cycle. All we have to
do is then to check which representation in $Hom(K_{3},\Sigma )$ does extend
to $K_{4}$. To this end, observe that $K_{4}$ is gotten from $K_{3}$ by
adjunction of a generator $x_{3}$ subject to the relations
\begin{equation*}
z_{m}x_{3}z_{m+2}=x_{3}z_{m+1}x_{3};m\in
%TCIMACRO{\U{2124} }%
%BeginExpansion
\mathbb{Z}
%EndExpansion
.
\end{equation*}%
Hence we may proceed as follows. Take a cycle in $Hom(K_{3},\Sigma )$ (by
abuse of language, i.e. identify each representation with its orbit, since a
representation in $Hom(K_{3},\Sigma )$ extends to $K_{4}$ if and only if
every element in its orbit does), and choose if possible a value $b_{3}\in
\Sigma $ for $\rho \left( x_{3}\right) $. This value must satisfy the
relations
\begin{equation*}
a_{m}b_{3}a_{m+2}=b_{3}a_{m+1}b_{3};
\end{equation*}%
for $m=0,\cdot \cdot \cdot ,p-1$, where $p$ is the cycle's length and the
indexation is $modp$. Observe that the choice $b_{3}=1$ is convenient, so
all cycles extend to $K_{4}$. However, this is the only possibility for type
I cycles to extend, for if $b_{3}$ commute with some $a_{m}$, then $b_{3}=1$.

\begin{proposition}
Let $\rho \in Hom\left( K_{4},\Sigma \right) $ be encoded by a cycle $\left(
a_{m}\right) _{m=0,...,p-1}$ and an element $b_{3}\in \Sigma $. Then $%
b_{3}^{p}=1$.
\end{proposition}

As a consequence, the order of $b_{3}$ divides $p$; hence, if $\gcd \left(
p,\left\vert \Sigma \right\vert \right) =1$, then $b_{3}=1$.

\bigskip

Before giving the general procedure, let us proceed one further step to show
that all type I cycles will vanish for $n\geq 5$. Take a cycle in $%
Hom(K_{3},\Sigma )$, along with a convenient value $b_{3}$ of $\rho \left(
x_{3}\right) $. We look for an element $b_{4}\in \Sigma $ satisfying :
\begin{equation*}
\begin{array}{cc}
a_{m}b_{4}=b_{4}a_{m+1}, & m=0,\cdot \cdot \cdot ,p-1;\text{and } \\
b_{3}b_{4}b_{3}=b_{4}b_{3}b_{4} &
\end{array}%
\end{equation*}%
Hence, if
\begin{equation*}
b_{3}b_{4}=b_{4}b_{3},
\end{equation*}%
then
\begin{equation*}
b_{3}=b_{4}.
\end{equation*}
Using%
\begin{equation*}
a_{m}b_{3}a_{m+2}=b_{3}a_{m+1}b_{3},
\end{equation*}%
and%
\begin{equation*}
a_{m}b_{4}=b_{4}a_{m+1},
\end{equation*}
we get
\begin{equation*}
b_{3}=b_{4}=1,
\end{equation*}%
and
\begin{equation*}
a_{m}=a_{m+1},\forall m=0,\cdot \cdot \cdot ,p-1,
\end{equation*}%
so that the representation is trivial. So, except for the trivial cycle,
only type II cycles with $b_{3}\neq 1$ possibly extend to $K_{5}$, with $%
b_{4}$ not commuting with $b_{3}$, in particular, $b_{4}\neq 1$. Note that $%
b_{3}$ and $b_{4}$ are conjugate.

\begin{algorithm}
The general procedure for $Hom(K_{n},\Sigma )$, $n\geq 5$ is to consider
only type II cycles along with convenient non trivial values $b_{3},\cdot
\cdot \cdot ,b_{n-2}$, which correspond to representations in $%
Hom(K_{n-1},\Sigma )$ and find a non trivial element $b_{n-1}\in \Sigma $
such that the following relations are satisfied:%
\begin{equation*}
\begin{array}{cc}
a_{m}b_{n-1}=b_{n-1}a_{m+1,} & m=0,\cdot \cdot \cdot ,p-1; \\
b_{n-1}b_{i}=b_{i}b_{n-1} & i=3,\cdot \cdot \cdot ,n-3; \\
b_{n-1}b_{n-2}b_{n-1}=b_{n-2}b_{n-1}b_{n-2} &
\end{array}%
\end{equation*}
\end{algorithm}

The element $b_{n-1}$ has to be non trivial, otherwise the representation is
trivial.

\bigskip

Using the various relations between the $a_{m}$'s and the $b_{i}$'s, we get
the following ones :

\begin{proposition}
\bigskip Let $\rho \in Hom(K_{n},\Sigma )$, $n\geq 5$ be encoded by a type
II cycle $\left( a_{m}\right) _{m=0}^{p-1}$ of length $p$, along with
elements $b_{3},...,b_{n-1}\in \Sigma $ as previously. Then we have the
following relations :
\end{proposition}

\begin{enumerate}
\item The $b_{i}$'s are non trivial and conjugate for $i=3,...,n-1$ and $%
\left[ b_{i}^{p},a_{m}\right] =1$, for $m=0,...,p-1;$ $i=4,...,n-1$.

\item $\left\{
\begin{array}{cc}
\left[ b_{i},b_{j}\right] =1; & \left\vert i-j\right\vert \geq 2; \\
\left[ b_{i},b_{j}\right] \neq 1; & \left\vert i-j\right\vert =1.%
\end{array}%
\right. $.

\item $b_{i}\neq b_{j}$, for $i,j=3,...,n-1$ and $i\neq j$, except for the
possibility $b_{3}=b_{5}$, in which case $\rho $ doesn't extend to $%
Hom(K_{7},\Sigma )$.
\end{enumerate}

As a consequence of the first relation, $p$ divides the order of $b_{i}$, $%
i=4,...,n-1$; hence that of $\Sigma $.

\section{The abelian case}

Starting with $n=3$, we see that the relation betwen the $a_{m}$'s, written
additively, becomes :%
\begin{equation*}
a_{m+2}=a_{m+1}-a_{m}.
\end{equation*}%
Implementing our algorithm gives a matix $A$ :
\begin{equation*}
\left( a,b\right) \overset{A}{\longmapsto }\left( b,b-a\right) ,
\end{equation*}%
with $A^{3}=-I$. All cycles have length dividing $6$. More precisely, the
length may be $1$ ($a=b=0$) or $2$ ($a=-b$, with $3a=0$) or $3$ ($2a=2b=0$)
or $6$ (otherwise). Moving to $n=4$, we find that all cycles extend to $%
K_{4} $ with $b_{3}=1$ (only). No non trivial cycle extends to $K_{5}$, that
is $Hom(K_{n},\Sigma )=0$, for $\Sigma $ abelian and $n\geq 5$. This latter
fact can also be seen from :
\begin{equation*}
\frac{K_{n}}{\left[ K_{n},K_{n}\right] }=1;n\geq 5.
\end{equation*}

\section{Extension to the braid group}

In this section, we address the question of extending representations
\begin{equation*}
\rho \in Hom(K_{n},\Sigma )
\end{equation*}%
to representations%
\begin{equation*}
\tilde{\rho}\in Hom(B_{n},\Sigma ).
\end{equation*}%
Applying [SiWi1 (3.5)], the extension is possible if and only if there is an
element $c\in \Sigma $ such that
\begin{equation*}
\begin{array}{cc}
a_{m}c=ca_{m+1,} & m=0,\cdot \cdot \cdot ,p-1; \\
cb_{i}=b_{i}c & i=3,\cdot \cdot \cdot ,n-1;%
\end{array}%
\end{equation*}%
Observe that such an element $c$ must satisfy $c\neq 1$, unless $\rho $ is
trivial; and $\left[ c^{p},a_{m}\right] =1$, for $m=0,...,p-1$. In
particular, $p$ divides the order of $c$. Hence, if $p\nmid \left\vert
\Sigma \right\vert $, then $\rho $ doesn't extend to $B_{n}$. Observe also
that a necessary condition for a representation $\rho \in Hom(K_{n},\Sigma )$
to extend to $\tilde{\rho}\in Hom(B_{n},\Sigma )$ is that $\rho \in
Hom(K_{n},\left[ \Sigma ,\Sigma \right] )$. A sufficient condition is that $%
\rho $ be the restriction of some representation $\hat{\rho}\in
Hom(K_{n+2},\Sigma )$ for if this is the case, the choice $c=b_{n+1}$ will
do. In this case, since $\hat{\rho}$ maps $K_{n}$ into $\left[ \Sigma
,\Sigma \right] $; it also maps $K_{n+2}$ into $\left[ \Sigma ,\Sigma \right]
$, for the $b_{i}$'s are conjugate (if $n\geq 4$) and \ $\left[ \Sigma
,\Sigma \right] $ is normal in $\Sigma .$ As a result, we get the following

\begin{proposition}
for $n\geq 6$, $Hom(K_{n},\Sigma )=Hom(K_{n},\Sigma ^{\left( r\right) })$,
where $r\geq 0$ is the smallest integer such that $\Sigma ^{\left(
r+1\right) }=\Sigma ^{\left( r\right) }$.
\end{proposition}

As a consequence, we get the fact that $n\geq 6$, $Hom(K_{n},\Sigma )$ is
trivial for any solvable group $\Sigma $. This generalizes the abelian case.
In this case, every $\rho \in Hom(B_{n},\Sigma )$ has a cyclic image
generated by $\rho \left( \sigma _{1}\right) =...=\rho \left( \sigma
_{1}\right) $, i.e. $\left\vert Hom(B_{n},\Sigma )\right\vert =\left\vert
\Sigma \right\vert $.

Actually we can enhance our algorithm to one which gives for fixed $n\geq 5$
the sets $Hom(K_{n},\Sigma )$ and $Hom(B_{n},\Sigma )$.

\textit{Step one: find all cycles of both types. This gives }$%
Hom(K_{3},\Sigma )$\textit{.}

\textit{Step two: For the trivial cycle, take any }$c\in \Sigma $\textit{\
to be arbitrary. For a type II cycle }$C$\textit{\ (with length }$p$\textit{%
\ dividing }$\left\vert \Sigma \right\vert $\textit{), find }$c\neq 1$%
\textit{\ such that }$a_{m}c=ca_{m+1}$\textit{, for }$m=0,\cdot \cdot \cdot
,p-1$\textit{. This gives }$Hom(B_{3},\Sigma )$\textit{.}

\textit{Step three: For a cycle of any type , take }$b_{3}=1$\textit{. For a
type II cycle }$C$\textit{\ (with length }$p$\textit{\ such that }$\gcd
\left( p,\left\vert \Sigma \right\vert \right) \neq 1$\textit{), find }$%
b_{3}\neq 1$\textit{\ such that }$a_{m}b_{3}a_{m+2}=b_{3}a_{m+1}b_{3}$%
\textit{, for }$m=0,\cdot \cdot \cdot ,p-1$\textit{. This gives }$%
Hom(K_{4},\Sigma )$\textit{.}

\textit{Step four: Beside the trivial cycle, take a type II cycle }$C$%
\textit{\ along with a convenient }$b_{3}$\textit{\ (with length }$p$\textit{%
\ dividing }$\left\vert \Sigma \right\vert $). \textit{\ If this cycle
occurs in }$Hom(B_{3},\Sigma )$\textit{\ with some convenient }$c$\textit{\
then :}

\textit{if }$cb_{3}=b_{3}c$ (which need not be checked if $b_{3}=1$),
\textit{then the representation }$\left[ C,b_{3}\right] $\textit{\ moves up
to a representation }$\left[ C,b_{3},c\right] $\textit{\ \ in }$%
Hom(B_{4},\Sigma )$\textit{;}

\textit{if }$cb_{3}c=b_{3}cb_{3}$\textit{\ }(which is impossible if $b_{3}=1$%
), \textit{\ then the representation }$\left[ C,b_{3}\right] $\textit{\
moves up to a representation }$\left[ C,b_{3},b_{4}\right] $\textit{\ \ in }$%
Hom(K_{5},\Sigma )$\textit{\ by taking }$b_{4}=c$\textit{.}

$\mathit{\cdot }$

$\mathit{\cdot }$

$\mathit{\cdot }$

\textit{Step i: take a representation }$\rho $\textit{\ in }$%
Hom(K_{i},\Sigma )$\textit{, encoded by a type II cycle }$C$\textit{\ along
with convenient values }$b_{3},\cdot \cdot \cdot ,b_{i-1}$\textit{. if }$%
\left[ C,b_{3},b_{4},\cdot \cdot \cdot ,b_{i-2},c\right] $\textit{\ occurs
in }$Hom(B_{i-1},\Sigma )$\textit{\ with some convenient }$c$\textit{\ then:}

\textit{if }$cb_{i-1}=b_{i-1}c_{\text{ }}$\textit{, then the representation }%
$\left[ C,b_{3},b_{4},\cdot \cdot \cdot ,b_{i-1}\right] $\textit{\ moves up
to a representation }$\left[ C,b_{3},b_{4},\cdot \cdot \cdot ,b_{i-1},c%
\right] $\textit{\ \ in }$Hom(B_{i},\Sigma )$\textit{;}

\textit{if }$cb_{i-1}c=b_{i-1}cb_{i-1}$\textit{, then the representation }$%
\left[ C,b_{3},b_{4},\cdot \cdot \cdot ,b_{i-1}\right] $\textit{\ moves up
to a representation }$\left[ C,b_{3},b_{4},\cdot \cdot \cdot ,b_{i-1},b_{i}%
\right] $\textit{\ \ in }$Hom(K_{i+1},\Sigma )$\textit{\ by taking }$b_{i}=c$%
\textit{.}

\section{\protect\bigskip Permutation representations}

Our goal in this section is to study representations of $K_{n}$ into the
symmetric group $S_{r}.$ Note that there is a natural homomorphism:%
\begin{equation*}
\pi :B_{n}\rightarrow S_{n},
\end{equation*}%
for all $n\geq 2$, given by $\sigma _{i}\longmapsto (ii+1)$. This restricts
to $K_{n}$ to give a non trivial homomorphism:%
\begin{equation*}
\begin{array}{cccc}
x_{i} & \longmapsto  & \left( 12\right) \left( ii+1\right)  & i=3,\cdot
\cdot \cdot ,n-1 \\
z_{m} & \longmapsto  & \left\{
\begin{array}{c}
\left( 132\right) , \\
\left( 123\right) ,%
\end{array}%
\right.  &
\begin{array}{c}
m\text{ is even} \\
m\text{ is odd}%
\end{array}%
\end{array}%
\end{equation*}

We start, as usual, with $n=3$, and describe $Hom(K_{3},S_{r})$. Note that
it contains $\pi \left\vert _{K_{3}}\right. $, for $r\geq 3$. Recall that a
non trivial element $a\in S_{r}$ has order two if and only if it is a
product of disjoint transpositions. Let $n_{r}$ be the number of such
elements. This gives us a means to compute the number of type I cycles to be
$\frac{1}{2}\left( 1+n_{r}+r!\right) $ and of representations coming from
them to be $3r!-2$.

Since $S_{2}=\left\{ 1,\left( 12\right) \right\} $, $Hom\left(
K_{3},S_{2}\right) $ consists only of the following type I cycle:
\begin{equation*}
\left( 1,\left( 12\right) \right) \rightarrow \left( \left( 12\right)
,\left( 12\right) \right) \rightarrow \left( \left( 12\right) ,1\right)
\rightarrow \left( 1,\left( 12\right) \right) ,
\end{equation*}%
along with the trivial representation. So

\begin{claim}
$\left\vert Hom\left( K_{3},S_{2}\right) \right\vert =4$.
\end{claim}

As for $Hom\left( K_{3},S_{3}\right) $, there are three type I cycles of
length $3$ corresponding to transpositions and one type I cycle of length $6$
corresponding to the $3$-cycle $(123)$ (and its inverse). Looking at type II
cycles, we find two cycles of length $9$ corresponding to the pairs $\left(
\left( 23),(12\right) \right) $ and $\left( \left( 23\right) ,\left(
123\right) \right) $ and one cycle of length $2$ corresponding to the pair $%
\left( \left( 123\right) ,\left( 132\right) \right) $. This last one is
exactly the orbit (under the shift map $\sigma $) of $\pi \left\vert
_{K_{3}}\right. $. All by all, we have

\begin{claim}
$\left\vert Hom\left( K_{3},S_{3}\right) \right\vert =36$.
\end{claim}

In the last section, we present among other things, the results of computer
calculations of type II cycles in the graph of $Hom\left( K_{3},S_{4}\right)
$ using Maple. Moving to $n=4$, we find that all (type I) cycles in $%
Hom\left( K_{3},S_{2}\right) $ extend to $K_{4}$ with $b_{3}=1$. So

\begin{claim}
$\left\vert Hom\left( K_{4},S_{2}\right) \right\vert =4$.
\end{claim}

No cycle in $Hom(K_{3},S_{3})$ extends to $K_{4}$ with non trivial $b_{3}$ :

\begin{claim}
$\left\vert Hom\left( K_{4},S_{3}\right) \right\vert =\left\vert Hom\left(
K_{3},S_{3}\right) \right\vert =36$.
\end{claim}

Out of $71$ (type II) cycles in $Hom(K_{3},S_{4})$\ only ten do extend to $%
K_{4}$, each with three possibilities for $b_{3}$ (the same for all; see the
last section). So

\begin{claim}
$\left\vert Hom\left( K_{4},S_{4}\right) \right\vert =\left\vert Hom\left(
K_{3},S_{4}\right) \right\vert +30$.
\end{claim}

For $n=5$, we find that no type I cycle and no type II cycle with $b_{3}=1$
extends to $K_{5}$, and that

\begin{claim}
$\left\vert Hom\left( K_{5},S_{2}\right) \right\vert =\left\vert Hom\left(
K_{5},S_{3}\right) \right\vert =1$.
\end{claim}

As for type II cycles with $b_{3}\neq 1$, none of the thirty cycles extends
to $K_{5}$:

\begin{claim}
$\left\vert Hom\left( K_{5},S_{4}\right) \right\vert =1$.
\end{claim}

We also find :

\begin{claim}
$\left\vert Hom\left( K_{6},S_{5}\right) \right\vert =\left\vert Hom\left(
K_{7},S_{6}\right) \right\vert =1$.
\end{claim}

\bigskip

Experimental results recover the well known fact that that the process given
by Algorithm 1 will stop at step $n=r$. That is:

\begin{proposition}
$Hom(K_{n},S_{r})$ is trivial for $r\geq 4$ and $n\geq r+1$.
\end{proposition}

\begin{proof}
see Lin
\end{proof}

It is obvious that a cycle (of any type) can not extend to $K_{n}$ if it
doesn't extend to $K_{n-1}$\ , so Proposition 5 \ asserts exactly that $%
Hom(K_{r+1},S_{r})$ is trivial. Recall that for $n\leq r$, $Hom(K_{n},S_{r})$
is not trivial since it contains the homomorphism $\pi \left\vert
_{K_{n}}\right. :K_{n}\rightarrow S_{n}$.

\begin{proposition}
for $n\geq 6$, $Hom(K_{n},S_{r})=Hom(K_{n},A_{r})$.
\end{proposition}

\begin{proof}
Apply Proposition 4 \ \ to $\Sigma =S_{r}$.
\end{proof}

\qquad Note that according to Proposition 5, for $n\geq r+1$ , every
representation $\tilde{\rho}:B_{n}\rightarrow S_{r}$ factorizes through the
abelianized group $\left( B_{n}\right) _{ab}$, and has a cyclic image.
Hence, there are $r!$ possible choices for $\tilde{\rho}$.

\section{Consequences}

Regarding the correspondence between subgroups of finite index of a group $K$
and its representations into symmetric groups, we investigate the subgroups
of index $r$ of $K_{n}$ for low degrees $r$. The general principle is to
compute the number of transitive representatations of $K$ into $S_{r}$ to
deduce the number of subgroups of $K$ with index exactly $r$. We start with $%
K_{3}$ as usual. Note that since $K_{3}$ is freely generated by $z_{0}$ and $%
z_{-1}$, it maps onto any symmetric group, and hence, has subgroups of every
index. Now, if a representation in $Hom(K_{3},S_{r})$ is transitive, then so
are the representations in its orbit. Consider a type I cycle in $%
Hom(K_{3},S_{r})$. Then the representations it defines are transitive if and
only if the defining element $a$ is (with respect to the action of $S_{r}$
on $\left\{ 1,\cdot \cdot \cdot ,r\right\} $). This exactly means that $a$
is an $r$-cycle. If $r>2$ then $a^{2}\neq 1$ and the cycle has length $6$.

\begin{claim}
The number of transitive representations $\rho \in $ $Hom(K_{3},S_{r})$, $%
r\geq 2$ coming from type I cycles is $3\left( r-1\right) !$.
\end{claim}

For $r=2$, there are only type I cycles and there is only one $2$-cycle,
which has length $3$; Hence, The number of transitive representations $\rho
\in $ $Hom(K_{3},S_{2})$ is $3$. The kernels of these representations give
rise to subgroups of $K_{3}$ with index $2$.

\begin{claim}
There are three subgroups of $K_{3}$ with index $2$.
\end{claim}

Now we compute the number of subgroups of $K_{3}$ with index $3$. Among all
representations we have seen in example 2, there are six transitive
representations coming from the only type I cycle and all representations
coming from type II cycles are transitive. Hence:

\begin{claim}
The number of transitive representations in $Hom\left( K_{3},S_{3}\right) $
is $26$, consequently there are thirteen subgroups of $K_{3}$ with index $3$.
\end{claim}

We can proceed in this way for every degree $r$. To compute the number of
transitive representations of $K_{3}$ into $S_{r}$ , we need only consider
those coming from type II cycles, since we already know the number of those
coming from type I cycles. This can be done using a computer algebra system,
by taking any cycle $C=(a_{0},\cdot \cdot \cdot ,a_{p-1})$ and checking if
the subgroup $\langle a_{0},\cdot \cdot \cdot ,a_{p-1}\rangle $ of $S_{r}$
acts transitively on $\left\{ 1,\cdot \cdot \cdot ,r\right\} $. If so, this
gives rise to $p$ transitive representations in $Hom\left(
K_{3},S_{3}\right) $. Then we divide the total number by $\left( r-1\right)
! $ to find the number of subgroups of $K_{3}$ of index $r$.

Now let us consider $Hom\left( K_{4},S_{r}\right) $. For $r=2$ we have, as
previously:

\begin{claim}
There are three subgroups of $K_{4}$ with index $2$.
\end{claim}

As for transitive representations in $Hom\left( K_{4},S_{3}\right) $, since
all cycles in $Hom\left( K_{3},S_{3}\right) $ extend to $K_{4}$ with only $%
b_{3}=1$, we have:

\begin{claim}
There are twenty six transitive representations in $Hom\left(
K_{4},S_{3}\right) $, hence thirteen subgroups of $K_{3}$ with index $3$.
\end{claim}

For $r\geq 4$, we have $3\left( r-1\right) !$ transitive representations
coming from type I cycles, and we must check which representation coming
from a type II cycle is transitive. For a cycle $C=(a_{0},\cdot \cdot \cdot
,a_{p-1})$ such that $\langle a_{0},\cdot \cdot \cdot ,a_{p-1}\rangle $
failed to be transitive, we check if $\langle a_{0},\cdot \cdot \cdot
,a_{p-1},b_{3}\rangle $ \ (with $b_{3}$ non trivial) is transitive. Indeed,
if $\langle a_{0},\cdot \cdot \cdot ,a_{p-1}\rangle $ is transitive, then so
is $\langle a_{0},\cdot \cdot \cdot ,a_{p-1},b_{3}\rangle $. Finally, we
divide the total number by $\left( r-1\right) !$ to find the number of
subgroups of $K_{4}$ of index $r$.

Now, we consider $n\geq 5$, where we get rid of type I cycles. Suppose we
have found the transitive representations in $Hom\left( K_{n-1},S_{r}\right)
$. We then take, for fixed $r$, a type II cycle $C=(a_{0},\cdot \cdot \cdot
,a_{p-1})$ along with values $b_{3},\cdot \cdot \cdot ,b_{n-1}$, such that $%
\langle a_{0},\cdot \cdot \cdot ,a_{p-1},b_{3},\cdot \cdot \cdot
,b_{n-2}\rangle $ \ failed to be transitive and check if $\langle
a_{0},\cdot \cdot \cdot ,a_{p-1},b_{3},\cdot \cdot \cdot ,b_{n-1}\rangle $
is transitive. We may enhance algorithm 1 by checking, each time we get a
new type II cycle, if it is transitive, and if not, we re-check at each time
the cycle extends from $K_{i}$ to $K_{i+1}$, $i=3,\cdot \cdot \cdot ,n-1$,
after having augmented it with $b_{i}$. Dividing by $\left( r-1\right) !$
the number of transitive representations in $Hom\left( K_{n},S_{r}\right) $
we find the number of subgroups of $K_{n}$ with index $r$. \ As a
consequence of Proposition 5, we get the following:

\begin{proposition}
For $n\geq 5$ and $2\leq r\leq n-1$, there are no subgroups of $K_{n}$ with
index $r$. Moreover, every nontrivial representation $\rho $ of $\ K_{n}$
into $S_{n}$ is transitive.
\end{proposition}

\begin{remark}
We can likewise investigate the number of subgroups of $B_{n}$ with a given
index $r$ by looking at transitive representations of $B_{n}$ into $S_{r}$.
Namely, according to Proposition 5, there is exactly one subgroup of index $%
r $ in $B_{n}$, for \ $1\leq r\leq n-1$. Moreover, if $\rho
:B_{n}\rightarrow S_{n}$ is a representation, then $\rho \left\vert
_{K_{n}}\right. $ is either trivial or transitive, according to Proposition
7. In the first case, $\rho $ has a cyclic image and we know when it is
transitive. In the second case, $\rho $ is transitive.
\end{remark}

\section{Experimental facts}

In what follows, we list the type II cycles of \ $Hom(K_{n},S_{r})$ for
various (small) $n$ and $r$. A word about the notation: each cycle $B\left[
a_{0},a_{1}\right] =[a_{2},a_{3},\cdot \cdot \cdot ,a_{p-1},a_{0},a_{1}]$ is
indexed by its first vertex $\left( a_{0},a_{1}\right) $ and is followed by
its length $p$. Elements $\tau \in S_{r}$ are ordered \ from $1$ to $r!$ \
with repect to the lexicographic order on the vectors $\left( \tau \left(
1\right) ,\cdot \cdot \cdot ,\tau \left( r\right) \right) $. It would have
taken too much space to list the cycles for $r=5$. We found that there were
no (type II) cycles in $Hom(K_{5},S_{4})$ nor in $Hom(K_{6},S_{5})$.
Furthermore, $Hom(K_{4},S_{3})$ contains no type II cycles with non trivial $%
b_{3}$, as predicted by Proposition 5.

\bigskip $n=3;r=3:$

B[2, 3] = [5, 6, 2, 5, 3, 6, 5, 2, 3]

9

B[2, 4] = [6, 3, 4, 2, 6, 4, 3, 2, 4]

9

B[4, 5] = [4, 5]

2

\bigskip

$n=3;r=4:$

B[2, 3] = [5, 6, 2, 5, 3, 6, 5, 2, 3]

9

B[2, 4] = [6, 3, 4, 2, 6, 4, 3, 2, 4]

9

B[2, 7] = [8, 2, 7]

3

B[2, 8] = [7, 2, 8]

3

B[2, 9] = [11, 6, 16, 18, 3, 20, 19, 2, 9]

9

B[2, 10] = [12, 3, 23, 21, 6, 14, 13, 2, 10]

9

B[2, 11] = [9, 6, 18, 16, 3, 19, 20, 2, 11]

9

B[2, 12] = [10, 3, 21, 23, 6, 13, 14, 2, 12]

9

B[2, 13] = [19, 22, 4, 23, 15, 12, 11, 2, 13]

9

B[2, 14] = [20, 15, 18, 5, 22, 10, 9, 2, 14]

9

B[2, 15] = [21, 22, 2, 21, 15, 22, 21, 2, 15]

9

B[2, 16] = [22, 15, 16, 2, 22, 16, 15, 2, 16]

9

B[2, 17] = [23, 7, 24, 23, 2, 17]

6

B[2, 18] = [24, 7, 18, 17, 2, 18]

6

B[2, 19] = [13, 22, 23, 4, 15, 11, 12, 2, 19]

9

B[2, 20] = [14, 15, 5, 18, 22, 9, 10, 2, 20]

9

B[2, 23] = [17, 7, 23, 24, 2, 23]

6

B[2, 24] = [18, 7, 17, 18, 2, 24]

6

B[3, 7] = [13, 15, 3, 13, 7, 15, 13, 3, 7]

9

B[3, 8] = [14, 22, 17, 14, 3, 8]

6

B[3, 9] = [15, 7, 9, 3, 15, 9, 7, 3, 9]

9

B[3, 10] = [16, 7, 11, 5, 15, 23, 20, 3, 10]

9

B[3, 11] = [17, 22, 11, 8, 3, 11]

6

B[3, 12] = [18, 15, 4, 14, 7, 21, 19, 3, 12]

9

B[3, 14] = [8, 22, 14, 17, 3, 14]

6

B[3, 16] = [10, 7, 5, 11, 15, 20, 23, 3, 16]

9

B[3, 17] = [11, 22, 8, 11, 3, 17]

6

B[3, 18] = [12, 15, 14, 4, 7, 19, 21, 3, 18]

9

B[3, 22] = [24, 3, 22]

3

B[3, 24] = [22, 3, 24]

3

B[4, 5] = [4, 5]

2

B[4, 8] = [20, 16, 17, 4, 13, 8, 16, 20, 17, 13, 4, 8]

12

B[4, 9] = [21, 20, 4, 9]

4

B[4, 10] = [22, 13, 18, 6, 21, 11, 7, 4, 10]

9

B[4, 11] = [23, 12, 19, 11, 13, 23, 19, 4, 11]

9

B[4, 12] = [24, 9, 16, 8, 12, 4, 24, 16, 9, 8, 4, 12]

12

B[4, 16] = [12, 9, 4, 16]

4

B[4, 17] = [9, 20, 24, 4, 21, 17, 20, 9, 24, 21, 4, 17]

12

B[4, 18] = [10, 13, 11, 18, 21, 10, 11, 4, 18]

9

B[4, 19] = [14, 21, 18, 19, 12, 14, 18, 4, 19]

9

B[4, 20] = [13, 16, 4, 20]

4

B[4, 22] = [18, 13, 6, 11, 21, 7, 10, 4, 22]

9

B[5, 7] = [14, 16, 6, 23, 9, 22, 19, 5, 7]

9

B[5, 8] = [13, 21, 24, 5, 20, 8, 21, 13, 24, 20, 5, 8]

12

B[5, 9] = [17, 12, 21, 8, 9, 5, 17, 21, 12, 8, 5, 9]

12

B[5, 10] = [18, 9, 14, 10, 20, 18, 14, 5, 10]

9

B[5, 12] = [16, 13, 5, 12]

4

B[5, 13] = [20, 21, 5, 13]

4

B[5, 14] = [19, 16, 23, 14, 9, 19, 23, 5, 14]

9

B[5, 16] = [24, 13, 12, 17, 16, 5, 24, 12, 13, 17, 5, 16]

12

B[5, 19] = [7, 16, 14, 6, 9, 23, 22, 5, 19]

9

B[5, 21] = [9, 12, 5, 21]

4

B[5, 23] = [11, 20, 10, 23, 16, 11, 10, 5, 23]

9

B[6, 7] = [20, 22, 6, 20, 7, 22, 20, 6, 7]

9

B[6, 8] = [19, 15, 24, 19, 6, 8]

6

B[6, 10] = [24, 15, 10, 8, 6, 10]

6

B[6, 12] = [22, 7, 12, 6, 22, 12, 7, 6, 12]

9

B[6, 15] = [17, 6, 15]

3

B[6, 17] = [15, 6, 17]

3

B[6, 19] = [8, 15, 19, 24, 6, 19]

6

B[6, 24] = [10, 15, 8, 10, 6, 24]

6

B[8, 17] = [24, 8, 17]

3

B[8, 18] = [23, 8, 23, 18, 8, 18]

6

B[8, 24] = [17, 8, 24]

3

B[9, 13] = [9, 13]

2

B[9, 18] = [11, 16, 19, 18, 20, 11, 19, 9, 18]

9

B[10, 14] = [12, 23, 10, 21, 14, 23, 13, 10, 14]

9

B[10, 17] = [10, 19, 17, 19, 10, 17]

6

B[11, 14] = [24, 14, 11, 24, 11, 14]

6

B[12, 20] = [12, 20]

2

B[16, 21] = [16, 21]

2

\bigskip

$n=4;r=4:$

ten cycles of length $2$ along with three values $b_{3}=8,17,24$.

[8, [4, 5], [4, 9], [4, 16], [4, 20], [5, 12], [5, 13], [5, 21], [9, 13],
[12, 20], [16, 21]]

[17, [4, 5], [4, 9], [4, 16], [4, 20], [5, 12], [5, 13], [5, 21], [9, 13],
[12, 20], [16, 21]]

[24, [4, 5], [4, 9], [4, 16], [4, 20], [5, 12], [5, 13], [5, 21], [9, 13],
[12, 20], [16, 21]]

\begin{acknowledgement}
I am grateful to Susan G. Williams for many helpful e-mail discussions. I
also wish to thank the students M. Menouer and Z. Ziadi for their help in
computer search.
\end{acknowledgement}

\bigskip

\bigskip

\bigskip

\bigskip

\bigskip

\bigskip

\bigskip

\end{document}